\documentclass[12pt]{amsart}

\usepackage{
amsfonts,
latexsym,
amssymb,
}

\newcommand{\labbel}{\label}

\newtheorem{theorem}{Theorem}

\newtheorem{proposition}[theorem]{Proposition} 
 
\newtheorem{corollary}[theorem]{Corollary}

\theoremstyle{definition}

\newtheorem{problem}[theorem]{Problem}

\theoremstyle{remark}
\newtheorem{remark}[theorem]{Remark}

\newcommand{\brfrt}{\hspace{0 pt}}

\DeclareMathOperator{\cf}{cf}

\newcommand{\m}{\mathfrak}

\begin{document}
 
\title[$H$-closed, $D$-pseudocompact]
{For Hausdorff spaces,\\ 
$H$-closed =  $D$-pseudocompact for all ultrafilters $D$}

\author{Paolo Lipparini} 
\address{Dipartimento di Matematica\\Viale dei Decreti Attuattivi Scientifici\\II Universit\`a di Roma (Tor Vergata)\\I-00133 ROME ITALY}
\urladdr{http://www.mat.uniroma2.it/\textasciitilde lipparin}

\thanks{We wish to express our gratitude to X. Caicedo, S.  Garc{\'{\i}}a-Ferreira
 and  {\'A} Tamariz-Mascar{\'u}a
 for stimulating discussions and correspondence.}  



\begin{abstract}
We prove that, for an arbitrary topological space $X$,
the following two conditions are equivalent:
(a) Every open cover of $X$ has a 
finite subset with dense union
(b) $X$ 
 is $D$-pseudocompact, for every ultrafilter $D$.

Locally, our result asserts that if $X$ is weakly initially $\lambda$-compact, 
and $2^ \mu \leq \lambda $, then $X$ is $D$-\brfrt pseudocompact, for every 
ultrafilter $D$ over any set of cardinality $ \leq \mu$.
As a consequence, if $2^ \mu \leq \lambda $, then the product of any family of
 weakly initially $\lambda$-compact spaces is 
weakly initially $\mu$-compact.
\end{abstract}

\maketitle

Throughout this note $\lambda$ and $\mu$ are infinite cardinals. 
No separation axiom is assumed, if not otherwise specified.
By a product of topological spaces we shall always mean the Tychonoff product.

The notion of weak initial $\lambda$-compactness has been 
introduced by  Z. Frol{\'{\i}}k \cite{F} under a different name
and subsequently studied by various authors. See, e.~g., 
 Stephenson and Vaughan \cite{SV}. See \cite[Remark 3]{tapp2} for further references about this and related notions.

For Tychonoff spaces, and for $D$ an ultrafilter over $ \omega$, the notion of 
$D$-pseudocompactness has been introduced by 
Ginsburg and Saks \cite{GS}. Their paper contains also 
significant applications. The notion
has been extensively studied by many authors in the setting of Tychonoff 
spaces, especially in connection with various orders on
$ \omega^*$. See, e.~g.,   \cite{GF1,HST,ST}  and further references there
for results and  related notions.
In the case of an  ultrafilter over an arbitrary cardinal, the notion of $D$-pseudocompactness has been introduced and studied in 
Garc{\'{\i}}a-Ferreira \cite{GF}.

In this note we show that  weak initial $\lambda$-compactness
and $D$-\brfrt pseudocompactness
 are tightly connected.
In fact, $D$-\brfrt pseudocompactness for every ultrafilter $D$ is
equivalent to  weak initial $\lambda$-compactness for every 
infinite cardinal $\lambda$. No separation axiom is needed to prove the equivalence.
As mentioned in the abstract, our result has a local version (Theorem \ref{dps} below).

The situation described in this note
has some resemblance with the connections
between initial $\lambda$-\brfrt compactness and $D$-compactness.
See, e.~g., the survey by
R. Stephenson  \cite{St} for definitions and results,
in particular, Section 3 therein. 
However, Remark \ref{conc} here points out a significant difference.

\smallskip

We now recall the relevant definitions.
A topological space is said to be \emph{weakly initially $\lambda$-compact} if and only if every open cover of cardinality at most $\lambda$ has a 
finite subset with dense union. 
Notice that, for Tychonoff  spaces, weak initial $ \omega $-compactness 
is well known to be equivalent to pseudocompactness. 

If $D$ is an ultrafilter 
over some set $I$, a topological space $X$ is said to be 
 \emph{$D$-pseudocompact}  if and only if
every $I$-indexed sequence of nonempty open sets of $X$ has some
$D$-limit point, where $x$ is called a \emph{$D$-limit point} of 
the sequence  $(O_i) _{i \in I} $ if and only if,
for every neighborhood $U$ of $x$ in $X$,
$\{ i \in I \mid U \cap O_i \not= \emptyset \} \in D$.

\begin{theorem} \labbel{dps}
If $X$ is a  weakly initially $\lambda$-compact topological space, 
and $2^ \mu \leq \lambda $, then $X$ is $D$-pseudocompact, for every 
ultrafilter $D$ over any set of cardinality $ \leq \mu$.
 \end{theorem} 

\begin{proof} 
Suppose by contradiction that 
$X$ is  weakly initially $\lambda$-compact, $D$ is an
ultrafilter over $I$, $2 ^{|I|} \leq \lambda  $,
and $X$ is not $D$-pseudocompact.
Thus, there is a sequence
$(O_i) _{i \in I} $ of nonempty open sets of $X$ 
which has no $D$-limit point in $X$.
This means that, for every $x \in X$, there is an open  neighborhood 
$U_x$ of $x$ such that 
$\{ i \in I \mid U_x \cap O_i \not= \emptyset \} \not\in D$, that is, 
$\{ i \in I \mid U_x \cap O_i = \emptyset \} \in D$,
since $D$ is an ultrafilter. 
For each $x \in X$, choose some $U_x$ as above, and let
$Z_x= \{ i \in I \mid U_x \cap O_i = \emptyset \}$. Thus, $Z_x \in D$. 

For each $Z \in D$, let $V_Z = \bigcup \{U_x  \mid x \text{ is such that } Z_x=Z \} $. 
Notice that if $i \in Z \in D$, then 
$V_Z \cap O_i = \emptyset $. 
Notice also that  $(V_Z) _{Z \in D} $ is an open cover of $X$. 
Since $| D| \leq 2 ^{|I|} \leq \lambda $, then,
by   weak initial $\lambda$-compactness, there is a finite number
$Z_1$, \dots,  $Z_n$ of elements of $D$ such that   
$V _{Z_1} \cup  \dots \cup  V_ {Z_n}$ is dense in $X$.
Since $D$ is a filter, $Z= Z_1 \cap \dots \cap Z_n \in D$,
hence $ Z_1 \cap \dots \cap Z_n \not= \emptyset $.
Choose $i \in Z_1 \cap \dots \cap Z_n$. 
Then $O_i \cap V _{Z_1} = \emptyset $, \dots,  
$O_i \cap V _{Z_n} = \emptyset $, hence
$O_i \cap (V _{Z_1} \cup  \dots \cup  V_ {Z_n}) = \emptyset $, 
contradicting the conclusion that
$V _{Z_1} \cup  \dots \cup  V_ {Z_n}$
is dense in $X$, since, by assumption,
$O_i$ is nonempty. 
\end{proof}  

Theorem \ref{dps} shows that
 weak initial $\lambda$-compactness implies $D$-\brfrt pseudocompactness,
for ultrafilters over  sets of sufficiently small cardinality. 
The next proposition presents an easy 
result in the other direction.

Recall that an ultrafilter over $\mu$ is \emph{regular} if and only if
there is a family of $\mu$ elements of $D$ such that the intersection
of any infinite subset of the family is empty. 
As a consequence of the Axiom of Choice (actually, the Prime Ideal Theorem suffices), for every infinite cardinal $\mu$ there is a regular ultrafilter over $\mu$.

\begin{proposition} \labbel{15}  
If the topological space $X$ is $D$-pseudocompact, for some regular ultrafilter $D$ 
over $\mu$, then $X$ is weakly initially $\mu$-compact.
Actually, every power of $X$ is weakly initially $\mu$-compact.
 \end{proposition}

\begin{proof} 
E.~g., by \cite[Corollary 15]{tapp2}. 
\end{proof} 

\begin{corollary} \labbel{cor}
If $2^ \mu \leq \lambda $, then the product of any family of
 weakly initially $\lambda$-compact spaces is 
weakly initially $\mu$-compact.
 \end{corollary}

 \begin{proof} 
Choose some regular ultrafilter $D$ over $\mu$. 
Given any family of  weakly initially $\lambda$-compact spaces, then,
by Theorem \ref{dps}, each member of the family is $D$-pseudocompact. 
Since $D$-pseudocompactness is productive \cite{GS}, the product is
$D$-\brfrt pseudocompact, hence  weakly initially $\mu$-compact, because of the choice
of $D$, and by Proposition \ref{15}.
\end{proof}  

Let us say that a topological space is \emph{weakly initially $<\nu$-compact} if and only if every open cover of cardinality  $<\nu$ has a 
finite subset with dense union. That is,  weak initial $<\nu$-compactness
means  weak initially $ \lambda $-compactness for all $ \lambda  < \nu $. 
Recall that a topological space is said to be \emph{initially $ \lambda $-compact}
if and only if  every open cover of cardinality at most $ \lambda $ has a 
finite subcover.

\begin{corollary} \labbel{inacc}
Suppose that $\nu$ is a strong limit cardinal.
 \begin{enumerate}   
\item  
Any product of a family of 
weakly initially $<\nu$-compact topological spaces 
is weakly initially $<\nu$-compact.
\item  
If $\nu$ is singular, then a product of a family of topological spaces 
is weakly initially $\nu$-compact,
provided that each factor is both  
weakly initially $\nu$-compact  
and initially $ 2 ^{ \cf \nu }$-compact.
 \end{enumerate}
 \end{corollary} 

\begin{proof}
(1) is immediate from Corollary \ref{cor}, and the assumption that $\nu$ is
a strong limit cardinal.

(2) Suppose that we have a product as in the assumption. By (1),
the product is weakly initially $<\nu$-compact.
By known results, or by a variation on the proof of Theorem \ref{dps} 
 (see Remark \ref{conc} or Theorem \ref{dpsf}),
any product of  initially $ 2 ^{ \cf \nu }$-compact
spaces is initially $ \cf \nu $-compact.
(2) now follows from the easy fact that a 
weakly initially $<\nu$-compact
and initially $ \cf \nu $-compact space is 
weakly initially $\nu$-compact
(actually,  a weakly initially $<\nu$-compact
and  $ [\cf \nu , \cf \nu] $-compact space is 
weakly initially $\nu$-compact.)
 \end{proof}  

We now give the characterization of Hausdorff-closed
spaces announced in the title. Recall that
a topological space $X$ is said to be $H(i)$ if and only if 
every open filter base on X has nonvoid adherence.
Equivalently, a topological space is $H(i)$ if and only if  every open cover has a 
finite subset with dense union.
A Hausdorff space is \emph{$H$-closed}  (or \emph{Hausdorff-closed}, or  \emph{absolutely closed}) if and only if  it is 
closed in every Hausdorff space in which it is embedded.
It is well known that a Hausdorff 
topological space is $H$-closed if and only if
it is $H(i)$.   
A regular Hausdorff space is $H$-closed if and only if it is compact.
See, e.~g.,  \cite{ScSt} for references.

\begin{theorem} \labbel{hi}
For every topological space $X$, the following conditions are equivalent.
  \begin{enumerate}   
 \item 
$X$ is $H(i)$.
\item
$X$ is 
weakly initially $\lambda$-compact, for every infinite cardinal $\lambda$.
\item
$X$ is $D$-pseudocompact, for every ultrafilter $D$.
\item
For every infinite cardinal $\lambda$, there exists some
regular ultrafilter $D$ over $\lambda$ such that $X$ is $D$-pseudocompact.
  \end{enumerate}

If $X$ is Hausdorff (respectively, Hausdorff and regular) then the preceding conditions are also equivalent to, respectively:
  \begin{enumerate}    
\item[(5)]
$X$ is $H$-closed. 
\item[(6)]
$X$ is compact. 
   \end{enumerate}
\end{theorem}

 \begin{proof}
(1) and (2) are  equivalent, because of the above mentioned
characterization of $H(i)$ spaces.

(2) $\Rightarrow $  (3) is immediate from Theorem \ref{dps}.

(3) $\Rightarrow $  (4) follows from the fact that, as we mentioned
right before Proposition \ref{15}, for every infinite cardinal $\lambda$, there 
does exist some
regular ultrafilter  over $\lambda$.

(4) $\Rightarrow $  (2) follows from  Proposition \ref{15}.

The equivalences of (1) and (5), and of (1) and (6), under the respective
assumptions, follow from the remarks before the statement of the theorem.
 \end{proof}  

As a consequence of Theorem  \ref{hi}, we get another proof of some classical results.

\begin{corollary} \labbel{hiprod}
Any product of a family of $H(i)$ spaces is an $H(i)$ space. 

Any product of a family of $H$-closed Hausdorff spaces is  $H$-closed.
 \end{corollary}

 \begin{proof} 
By Theorem  \ref{hi}, and the mentioned result by Ginsburg and Saks
\cite{GS}  that
$D$-pseudocompactness is productive. 
\end{proof}   

\begin{remark} \labbel{conc}    
In conclusion, a few remarks are in order.
The situation described in this note is almost entirely similar
to the case dealing with  initial $\lambda$-\brfrt compactness and $D$-compactness.
Indeed, the proof of Theorem \ref{dps} can be easily modified in order
to show directly that if $2^\mu \leq \lambda $, then every initially $\lambda$-compact 
 topological space is 
 $D$-compact, for every ultrafilter $D$ over any cardinal $ \leq \mu$
(see also Theorem \ref{dpsf} and the remark thereafter). 
This result, however, is already an immediate consequence of 
implications (8) and (5) in \cite[Diagram 3.6]{St}.
Since $D$-compactness, too, is productive, we get
that
if $2 ^ \mu \leq \lambda $, then  any product of 
initially $\lambda$-compact 
 spaces is initially $\mu$-compact, 
  the result analogue to Corollary \ref{cor}.
The above arguments furnish also a proof of the well known result
that a space is compact if and only if it is $D$-compact, for every 
ultrafilter $D$, a theorem which, in turn, has 
the Tychonoff theorem 
that every product of compact spaces is compact as an immediate consequence.
This is entirely parallel  to Theorem \ref{hi} and Corollary \ref{hiprod}.
  
However, a subtle  difference exists between the two cases.
A sufficient condition for a topological space $X$ to be 
initially $\lambda$-compact is that, for every $\lambda'$
 with $ \omega\leq \lambda ' \leq \lambda $, there exists some
ultrafilter $D$ uniform over $\lambda'$ such that $X$ is $D$-compact
(see \cite[Theorem 5.13]{St} or, again, \cite[Diagram 3.6]{St}). 
The parallel statement fails, in general, for   
weak initial $\lambda$-compactness and $D$-pseudocompactness.
Indeed, under some set theoretical hypothesis, \cite[Example 1.9]{GF} 
constructed a space $X$ which is $D$-pseudocompact, for some ultrafilter 
uniform $D$ over $ \omega_1$, hence necessarily 
$D'$-\brfrt pseudocompact, for some ultrafilter $D'$ 
uniform over $ \omega$, but $X$ is not weakly initially 
$ \omega_1$-compact, actually, not even $ \omega_1$-pseudocompact. 
Cf. also \cite[Remark 30]{tapp2}. 

The above counterexample shows that, in our arguments, 
 and, in particular, in Proposition \ref{15}, 
we do need the notion 
of a regular ultrafilter; on the contrary, in the corresponding theory for
initial compactness, (a sufficient number of) uniform 
ultrafilters are enough. 
\end{remark}
\smallskip

Theorem \ref{dps} can be generalized to the abstract
framework of \cite[Section 5]{tapp2}. We recall here only the definitions,
and refer to \cite{tapp2} for motivations and further references.  

Suppose that $X$ is a topological space, $\mathcal F$ 
is a family of subsets of $X$, and 
  $\lambda$ is an infinite cardinals.
We say that
$X$ is
\emph{$\mathcal F$-$ [ \omega , \lambda ]$-\brfrt compact}
if and only if, 
for every open cover
 $( O _ \alpha ) _{ \alpha \in \lambda } $
 of $X$, there exists some finite $W \subseteq \lambda $
such that
$F \cap \bigcup _{ \alpha \in W}  O_ \alpha \not= \emptyset  $,
for every $F \in \mathcal F$.
If $D$ is an ultrafilter over some set $I$,
we say that $X$ is  $\mathcal F$-$D$-\emph{compact} 
if and only if every sequence
$(F_i)_{i \in I}$ 
of members of $\mathcal F$ 
has some $D$-limit point in $X$.

\begin{theorem} \labbel{dpsf}
If $X$ is an $\mathcal F$-$ [ \omega , \lambda ]$-\brfrt compact topological space, 
and $2^ \mu \leq \lambda $, then $X$ is $\mathcal F$-$D$-compact, for every 
ultrafilter $D$ over any set of cardinality $ \leq \mu$.
 \end{theorem} 

Theorem \ref{dpsf} is proved in a way similar
to Theorem \ref{dps},
by replacing everywhere the family 
$(O_i) _{i \in I} $   by an appropriate family $(F_i) _{i \in I} $
of members of $\mathcal F$.

 Notice that Theorem \ref{dps}
is the particular case of Theorem \ref{dpsf} when
$\mathcal F$ is the family of all nonempty open sets of $X$.
By considering  the particular case of Theorem \ref{dpsf}
in which $\mathcal F$ is the family of all singletons of $X$ we obtain  
the parallel result mentioned in Remark \ref{conc},
asserting that if $2^ \mu \leq \lambda $, then initial $\lambda$-compactness 
implies $D$-compactness, for every ultrafilter over a set of cardinality $\leq \mu$.

\begin{corollary} \labbel{corf}
Suppose that $X$ is a topological space, and $\mathcal F$ 
is a family of subsets of $X$. Then the following conditions are equivalent.
  \begin{enumerate}   
 \item  
$X$ is $\mathcal F$-$ [ \omega , \lambda ]$-\brfrt compact, for every infinite cardinal $\lambda$.
\item
$X$ is  $\mathcal F$-$D$-compact, for every ultrafilter
$D$.
\item
For every infinite cardinal $\lambda$, there exists some
regular ultrafilter $D$ over $\lambda$ such that $X$ is $\mathcal F$-$D$-compact.
 \end{enumerate}
 \end{corollary}  

\begin{proof}
Same as the proof of Theorem  \ref{hi}. The implication (3) $\Rightarrow $  (1)
follows from \cite[Theorem 35(2) $\Rightarrow $  (4)]{tapp2} with $| T|=1$.  
\end{proof}

As a concluding observation, we expect that Corollary \ref{cor} gives an optimal result,
 but we have not checked it.

\begin{problem} \labbel{prob}    
Characterize those pairs of cardinals
$\lambda$ and $\mu$ such that  the product of any family of
 (weakly) initially $\lambda$-compact spaces is 
(weakly) initially $\mu$-compact.
\end{problem}

\end{document}